\documentclass[10pt,a4paper,reqno]{amsart}
\usepackage[latin1]{inputenc}
\usepackage{amsmath,enumerate,amsthm,graphicx,color,verbatim}
\usepackage{amsfonts,amscd}
\usepackage{latexsym}
\usepackage{amssymb}
\usepackage{breqn}
\newtheorem{thm}{Theorem}

\newtheorem{defn}{Definition}

\newtheorem{lemma}{Lemma}

\newtheorem{prop}{Proposition}

\newcommand{\Z}{\mathbb{Z}}
\newcommand{\D}{\mathbb{D}}
\newcommand{\R}{\mathbb{R}}

\newcommand{\N}{\mathbb{N}}

\begin{document}
\title[Purely sc spectrum for subshift CMV operators]{Purely singular continuous spectrum for CMV operators generated by subshifts
}
\author{Darren C. Ong}
\address{Department of Mathematics, Rice University, Houston, TX~77005, USA}

\email{darren.ong@rice.edu}
\begin{abstract}
We prove uniform absence of point spectrum for CMV operators corresponding to the period doubling subshift. We also prove almost sure absence of point spectrum for CMV operators corresponding to a class of Sturmian subshifts. Lastly, we prove almost sure absence of point spectrum for CMV operators corresponding to some subshifts generated by a coding of a rotation. 
\end{abstract}
\thanks{The author was supported in part by NSF grant DMS--1067988}
\maketitle
\begin{section}{Introduction}
We consider the extended CMV operator $\mathcal E$ on the space $\ell^2(\mathbb Z)$ of complex square-summable sequences. $\mathcal E$ is given by
\begin{equation*}\label{CMVmatrix}
\left(
\begin{array}{ccccccc}
\ldots&\ldots& \ldots&\ldots&\ldots&\ldots&\ldots\\
\ldots&-\overline{\alpha(0)}\alpha(-1)&\overline{ \alpha(1)}\rho(0)&\rho(1)\rho(0)&0&0&\ldots\\
\ldots&-\rho(0)\alpha(-1)&- \overline{\alpha(1)}\alpha(0)&-\rho(1)\alpha(0)&0&0&\ldots\\
\ldots&0& \overline{\alpha(2)}\rho(1)&-\overline{\alpha(2)}\alpha(1)&\overline{\alpha(3)}\rho(2)&\rho(3)\rho(2)&\ldots\\
\ldots&0& \rho(2)\rho(1)&-\rho(2)\alpha(1)&-\overline{\alpha(3)}\alpha(2)&-\rho(3)\alpha(2)&\ldots\\
\ldots&0& 0&0&\overline{\alpha(4)}\rho(3)&-\overline{\alpha(4)}\alpha(3)&\ldots\\
\ldots&\ldots& \ldots&\ldots&\ldots&\ldots&\ldots\\
\end{array}
\right),
\end{equation*}
where $\{\alpha(n)\}_{n\in\Z}$ is a sequence in $\mathbb D$, the open unit disk, and $\rho(n)=\sqrt{1-\vert\alpha(n)\vert^2}$. The $\alpha(n)$ are known as Verblunsky coefficients. 

The CMV operator was developed in \cite{CMV} in connection with the study of orthogonal polynomials of the unit circle. The $\ell^2(\mathbb N)$ version of this operator captures the recurrence relation of the orthogonal polynomials, and thus allowed us to apply spectral theoretic tools to the study of orthogonal polynomials on the unit circle, much in the same way that Jacobi operators allowed us to use spectral theoretic tools in the theory of orthogonal polynomials on the real line. The two-volume text \cite{Simon} extensively explores this point of view. 

Indeed, we can view the CMV operator as a unitary analogue of the Jacobi operator. There has been a lot of work done in discovering CMV analogues for theorems in the theory of Jacobi or Schr\"odinger operators, and vice versa.

There has been recent interest in studying the case where the extended CMV operator is generated by subshifts. That is, given an alphabet $\mathcal A=\{a,b\}$, we consider a word $u\in \mathcal A^\N$. The \em hull \em $\Omega$ generated by this substitution is the set of all two-sided sequences all of whose subblocks are also subblocks of $u$. The \em subshift dynamical system \em $(\Omega,\mathcal T)$ consists of the hull, the left shift $\mathcal T$ on $\Omega$, and the product topology on $\Omega$. For each $\omega\in \Omega$, we can then generate a CMV operator $\mathcal E_\omega$ by taking $\alpha_\omega(n)=f(\omega_n)$, where $\omega_n$ refers to the $n$th letter of $\omega$, and $f:\mathcal A\to \mathbb D$ is a nonconstant function.

For example, given an interval $I\subseteq[0,1)$, irrational $\theta$, $\beta\in[0,1)$ and $g:\{0,1\}\to\mathcal A$, a \em rotation coding \em sequence is a sequence in $n\in\N$ generated by $u=g(\chi_I(\theta n+\beta \mod 1 ))$. If $I=[1-\theta,1)$, then the sequences are called \em Sturmian sequences\em. Furthermore, if $\theta$ and $\beta$ are both the fractional part of the golden mean (that is, $\theta=\beta=(\sqrt{5}-1)/2$), then the sequence is called the \em Fibonacci sequence\em.

For these cases, it is often useful to write $\theta$ in terms of a continued fraction expansion. That is,

\[\theta=a_0+\frac{1}{a_1+\frac{1}{a_2+\ddots}}.\]
 We also define $p_n/q_n$ to be the $n$th rational approximation of $\theta$. Then the $q_n$ are integers that satisfy
\begin{equation}\label{qrecursion}
q_{n+1}=a_nq_n+q_{n-1},
\end{equation}
with initial conditions $q_0=0, q_1=1$.

Another important class of subshifts are the \em substitution sequence subshifts\em. We define a substitution as a semigroup homomorphism $S:\mathcal A^*\to\mathcal A^*$, where $\mathcal A^*$ is the semigroup of words generated by $\mathcal A$ under the concatenation operation.  If $w=c_1\ldots c_l$ with $c_j\in\mathcal A$, then $S(w)=S(c_1)\ldots S(c_l)$. We assume that there is a distinguished element $a$ in $\mathcal A$ so that $S(a)$ starts with $a$ and that there exists a $k$ so $b\in S^k(a)$ (here $S^k$ denotes the compostion of $S$ with itself $k$ times). We also assume that $\vert S(a)\vert$, the length of the word $S(a)$ is more than $1$. Thus $S(a), S^2(a), S^3(a),\ldots $ all start with $a$ and we have a limit $S^{\infty}(a)\in \mathcal A^{\mathbb Z_+}$. If we extend $S$ by concatenation to act on infinite words, $S^\infty(a)$ is invariant under the action of $S$. We take this $S^\infty(a)$ as our sequence $u$. Section 10.5A of \cite{Simon} provides a more thorough explanation of these substitution sequences for the interested reader.

We will enumerate a few important examples. A \em period doubling sequence \em is a substitution sequence corresponding to $S(a)=ab, S(b)=aa$. A \em Thue-Morse sequence \em is generated by $S(a)=ab$, $S(b)=ba$. Also, an alternative way to generate a Fibonacci sequence is through the substitution $S(a)=ab, S(b)=a$.

The extended CMV operator generated by subshifts of the Fibonacci sequence has been used to study the classical Ising model and to model certain types of quantum walks (refer to \cite{DMY2} for a discussion on the various applications).

We prove uniform absence of point spectrum for the CMV operator generated by all period doubling subshifts, almost sure absence of point spectrum for a large class of Sturmian subshifts which include the Fibonacci subshift, and almost sure absence of point spectrum for a large class of subshifts generated by coding of a rotation. Explicitly, we show

\begin{thm}\label{PDtheorem}Let $\Omega$ be the hull of the period doubling subshift. For every $\omega\in\Omega$, the operator $\mathcal E_\omega$ is purely singular continuous.
\end{thm}

\begin{thm}\label{Kaminaga0}
Consider a Sturmian subshift, and suppose $\theta$ is chosen so that there exists an $N$ such that $a_n=1$ for all $n>N$ (The Fibonacci subshift satisfies this condition). Then the CMV operator corresponding to a Sturmian subshift generated by the rotation number $\theta$ is purely singular continuous for Lebesgue almost every phase $\beta$.
 \end{thm}

For the following two theorems, we note that the existence of $q_n$ even for arbitrarily large $n$ is equivalent to the existence of $a_m$ odd for arbitrarily large $m$.

 \begin{thm}\label{Kaminaga1}
 Consider a Sturmian subshift, and suppose $\theta$ is chosen so that there exists an infinite subset $\tilde {\mathbb N}$ of $\mathbb N$ so that $q_n$ is even for $n\in\tilde{\mathbb N}$. Assume further that $\limsup_{n\in \tilde N} a_n\geq 3$. Then the CMV operator corresponding to a Sturmian subshift generated by the rotation number $\theta$ is purely singular continuous for Lebesgue almost every phase $\beta$.
 \end{thm}

\begin{thm}\label{Kaminaga2}
 Consider a rotation coding subshift, where $I$ is an interval in $\R/\Z$. Suppose also that $\theta$ is chosen such that there exists an infinite subset $\tilde {\mathbb N}$ of $\mathbb N$ so that $q_n$ is even for $n\in\tilde{\mathbb N}$. Assume further that $\limsup_{n\in\tilde{\mathbb N}} a_n\geq 5$. Then the spectrum of the CMV operator corresponding to a subshift generated by the coding of the rotation $\theta$ is purely singular continuous for Lebesgue almost every  phase $\beta$.
 \end{thm}

These theorems are CMV analogues of theorems in \cite{Damanik.doubling} and \cite{Kaminaga}, which address these results in the setting of the discrete Schr\"odinger operator (that is, where the potential sequence of the Schr\"odinger operator is generated by these subshifts). Previously, the Schr\"odinger operator corresponding to the Fibonacci subshift was studied in \cite{Suto87} and \cite{Suto89}, and these arguments were then extended to the general Sturmian case in \cite{BIST}. The case for the period doubling Hamiltonian was addressed extensively in \cite{BBG}. The tools that we apply in the present paper were adapted from the ideas developed in these earlier works.

In general, Schr\"odinger operators generated by these subshifts have received a lot of attention. \cite{Suto95}, \cite{Damanik.subshiftexpo} and \cite{Damanik-Embree-Gorodetski} are three extensive surveys of the many results in this area. Much of the impetus behind this research is the fact that these subshifts have been accepted as one-dimensional models of quasicrystals (quasicrystals are structures that are ordered and yet aperiodic; their discovery by Dan Shechtman won him the 2011 Nobel Prize in Chemistry), and so the corresponding Schr\"odinger operators provided insights into quasicrystalline properties. The reason the Sturmian and period doubling sequences attracted interest in this way is that (in a certain precise sense) they are among the simplest possible aperiodic sequences. An extensive discussion of this point of view can be found in Section 2 of \cite{Damanik-Embree-Gorodetski}.

The theory in the Schr\"odinger context is more well-developed than in the CMV context. We will point out in particular \cite{DKL}, which proves uniformly purely singular continuous spectrum in the setting of Schr\"odinger operators generated by Sturmian subshifts, without the additional hypotheses in \cite{Kaminaga} or in our Theorem \ref{Kaminaga1}.
 
In addition, there has been very detailed understanding of the gaps in the spectrum for several classes of these subshift Schr\"odinger operators, as well as a lot of work on bounding the Hausdorff dimension of the associated spectral measures.

There has also been work applying some of these ideas (particularly with regard to the Fibonacci subshift) to the more general Jacobi settting. Examples of recent results in this area can be found in \cite{Damanik-Gorodetski}, \cite{Dahl} and \cite{Yessen.tridiagonal},
 
Let us now discuss known results in the CMV setting. Section 10.5A of \cite{Simon} surveys some of the developments and introduces \em Simon's subshift conjecture\em, which states that CMV operators corresponding to minimal, aperiodic subshifts had zero measure spectrum. This had been verified in \cite{DL_Szego} for a broad class of aperiodic subshifts that include the Sturmian, period doubling and rotation coding subshifts. Ultimately, the full conjecture was proven false in \cite{Avila-Damanik-Zhang}.

As for previous results concerning point spectrum, \cite{Ong-QP} proved generic absence of point spectrum for quasiperiodic shifts on a $d$-dimensional torus, and \cite{Ong-LP} proved generic absence of point spectrum for limt-periodic Verblunsky coefficients, which are generated by translations on a Cantor group (that is, a totally disconnected compact Abelian topological group with no isolated points). 

The Fibonacci case was studied extensively in \cite{DMY1} and \cite{DMY2}, in which the authors prove that the spectrum is a zero-measure Cantor set and prove results concerning the Hausdorff dimension of this set, and the nature of spectral gaps. They also prove absence of embedded point spectrum within the essential continuous spectrum, whereas we prove almost sure absence of point spectrum everywhere. They also discuss several applications of this result to quantum walks and Ising models (the connection between CMV matrices and quantum walks was first introduced in \cite{CMVG}). Beyond that, there has been little work done on these subshift CMV operators, which contrasts with the vast literature in the Schr\"odinger setting. 

A fundamental tool in proving a lot of the previously mentioned results are the \em transfer matrices \em related to the Schr\"odinger or CMV operators. The transfer matrices are a sequence of $2\times 2$ matrices that control the growth in the solutions of the eigenvalue equations of the Schr\"odinger or CMV operators, and in particular understanding the traces of the transfer matrices reveals a lot about the spectral measures of the corresponding operator. 

As an explicit example, we can write the time-independent one-dimensional discrete Schr\"odinger equation as 
\begin{equation}\label{SchrodingerEq}
u_{n+1}+V(n)u_n+u_{n-1}=Eu_n,
\end{equation}
where $E$, the energy, is a real constant.
This equation can also be expressed in the form

\[
\begin{pmatrix}
u_{n+1}\\
u_n
\end{pmatrix}
=
\begin{pmatrix}
E-V(n) & -1\\
1& 0
\end{pmatrix}
\begin{pmatrix}
u_{n}\\
u_{n-1}
\end{pmatrix}.
\]

We refer to the $2\times 2$ matrix in the previous equation as the \em Schr\"odinger transfer matrix \em. Notice that where $H$ is the Schr\"odinger operator with potential $V(n)$, then $u$ solves the eigenvalue equation $Hu=Eu$, and clearly the traces of the transfer matrices are related to the growth of the solutions $u$. 

In the CMV setting, the analogue of (\ref{SchrodingerEq}) is known as the \em Szeg\H o recurrence\em. Given $z\in\D$ and $n\in\mathbb N$,

\[ \rho(n)\varphi_{n+1}(z)=z\varphi_n(z)-\overline{\alpha(n)}\overline{\varphi_n(1/\overline z)}z^n,\]

which can be expressed in transfer matrix form as

\[
\begin{pmatrix}
\varphi_{n+1}(z)\\
\overline{\varphi_{n+1}(1/\overline z)}z^{n+1}
\end{pmatrix}
=
\rho(n)^{-1}
\begin{pmatrix}
z & -\overline{\alpha(n)}\\
- \alpha(n)z &1 
\end{pmatrix}
\begin{pmatrix}
\varphi_{n}(z)\\
\overline{\varphi_{n}(1/\overline z)}z^{n}
\end{pmatrix}.
\]
The $2\times 2$ matrix in this equation is known as the \em Szeg\H o transfer matrix\em. Here the $\varphi$ are the sequence of orthogonal polynomials on the unit circle associated with the one-sided infinite CMV operator $\mathcal C$. The polynomials are closely related to the solutions $u$ of $\mathcal C u=zu$, but it is not true that $u=\varphi$. In particular, if we consider  instead $\mathcal E$, the two-sided infinite CMV operator, we don't have associated orthogonal polynomials at all.

One key innovation in this paper is that we use transfer matrices introduced in \cite{GZ} (which we refer to as Gesztesy-Zinchenko transfer matrices), rather than the Szeg\H o transfer matrices that are more commonly used in the CMV setting. These Gesztesy-Zinchenko transfer matrices take the form

\[\frac{1}{\rho}\begin{pmatrix}
z&-\overline{\alpha}\\
-\alpha z&1
\end{pmatrix}.
\]

Szeg\H o transfer matrices do not directly allow us to understand the point spectrum of $\mathcal E$, whereas understanding the Gesztesy-Zinchenko transfer matrices do. We thus prove trace bounds, and develop a Gordon lemma and a Floquet theory for these Gesztesy-Zinchenko transfer matrices, and this enables us to apply the combinatorial arguments from \cite{Damanik.doubling} and \cite{Kaminaga}.

\begin{section}{Acknowledgements}
I wish to thank David Damanik and Paul Munger for helpful conversations. I also wish to thank the anonymous referees for many useful suggestions, corrections and comments.
\end{section}
\end{section}
\begin{section}{The CMV Gordon Lemma}
In this section, we describe our central tool, the Gordon lemma for CMV operators that appears in \cite{Ong-LP}. The Gordon lemma was initially developed for Schr\"odinger operators in \cite{Gordon}. In essence, the lemma exploits the existence of local repetitions to prove absence of point spectrum. A good survey of the many variants of the technique appears in \cite{Damanik.Gordon}. 

We modify the ideas of those papers to better suit our purposes. In particular, we use instead transfer matrices that are introduced in \cite{GZ}.

Specify $z\in\partial\D$. For integer $n$, let us first introduce the transfer matrix $T_n$ described in (2.18) of \cite{GZ}. It is defined by

\begin{equation}
T_n(z)=\begin{cases}
\frac{1}{\rho(n)}\begin{pmatrix}
-\overline{\alpha(n)} &z\\
1/z& -\alpha(n)
\end{pmatrix}, \text{ $n$ odd,}
\\
\frac{1}{\rho(n)}\begin{pmatrix}
-{\alpha(n)} &1\\
1& -\overline{\alpha(n)}
\end{pmatrix}, \text{ $n$ even.}
\end{cases}
\end{equation}

Note that the notation in \cite{GZ} differs slightly from ours. In particular, their $\alpha_n$ is our $-\overline{\alpha(n)}$.

Let us also define
\[\Theta_{n}=\begin{pmatrix}
\overline{\alpha(n)}&\rho(n)\\
\rho(n)&-\alpha(n)
\end{pmatrix}.
\]

For a solution $u$ to the CMV eigenvalue equation $\mathcal E u=zu$, there is a corresponding solution $v$ to the eigenvalue equation $\mathcal E^T v=zv$ that satisfies $\mathcal Mu =zv$, where $\mathcal M:\ell^2(\Z)\to\ell^2(\Z)$ is the direct sum of $2\times 2$ matrices of the form $\Theta_{2k-1}$ with $k$ running from $-\infty$ to $+\infty$. The existence of such $v$ is proven in Lemma 2.2 of \cite{GZ}.

The transfer matrices satisfy
\begin{equation}
\begin{pmatrix}
u_n(z)\\
v_n(z)
\end{pmatrix}
=
T_n(z)
\begin{pmatrix}
u_{n-1}(z)\\
v_{n-1}(z)
\end{pmatrix}
.
\end{equation}
Note that the $T_n$ have determinant $-1$. We also define 
\begin{equation}
M_n(z)= \begin{cases}
  T_n(z)\ldots T_1(z), n\geq 1,\\
  I,  n=0,\\
T_{n+1}(z)^{-1}\ldots T_0(z)^{-1}, n\leq -1.
\end{cases}
\end{equation}

The following version of the CMV Gordon lemma is adapted from the proof in the Schr\"odinger context given in Section 3 of \cite{Damanik.Gordon}. It differs from the version in \cite{Ong-LP}, which is based on the Gordon lemma in Section 10.2 of \cite{Cycon}.

\begin{prop}[Two-block CMV Gordon lemma]\label{2block}Fix a sequence of Verblunsky coefficients $\alpha$ and fix $z\in \partial \D$. Suppose there is an even whole number sequence $n_k\to\infty$ and a constant $C\geq 1$ so that for every $k$

\begin{equation}\alpha(j)=\alpha(j+n_k), 1\leq j\leq n_k \label{2blockcondition},
\end{equation}
 and also $\vert \mathrm{tr}M_{n_k}(z)\vert\leq C.$

Then $z$ is not an eigenvalue for $\mathcal E$.
\end{prop}
\begin{prop}[Three-block CMV Gordon lemma]\label{3block}Fix a sequence of Verblunsky coefficients $\alpha$ and fix $z\in \partial \D$. Suppose there is an even whole number sequence $n_k\to\infty$ so that for every $k$

 \begin{equation}\label{3blockcondition}
 \alpha(j-n_k)=\alpha(j)=\alpha(j+n_k), 1\leq j\leq n_k.
 \end{equation}

Then $z$ is not an eigenvalue for $\mathcal E$. 
\end{prop}

\begin{proof}[Proof of Proposition \ref{2block}]
By the Cayley-Hamilton theorem, for any even $n\geq 2$ we have the following matrix equation:
\[M_{n}(z)^2-\mathrm{tr}(M_{n}(z))M_{n}(z)+I=0.\]
 Note that the conditions we have on the Verblunsky coefficients imply that \[M_{n_k}(z)^2=M_{2n_k}(z).\] Thus we have for any $k$
\[ \begin{pmatrix}u_{2n_k}\\
v_{2n_k}
\end{pmatrix}-\mathrm{tr}(M_{n_k}(z))\begin{pmatrix}u_{n_k}\\
v_{n_k}
\end{pmatrix}+\begin{pmatrix}u_{0}\\
v_{0}
\end{pmatrix}=0.\]

Let us assume without loss of generality that $\vert u_0\vert^2+\vert v_0\vert ^2=1$. This implies that one of

\[ \begin{pmatrix}u_{2n_k}\\
v_{2n_k}
\end{pmatrix},\mathrm{tr}(M_{n_k}(z))\begin{pmatrix}u_{n_k}\\
v_{n_k}
\end{pmatrix},\]
has norm at least $1/2$. Stated another way, we have

\begin{align}&\max\left(\left\vert\left\vert\begin{pmatrix}u_{2n_k}\\
v_{2n_k}
\end{pmatrix}\right\vert\right\vert,\left\vert\left\vert\begin{pmatrix}u_{n_k}\\
v_{n_k}
\end{pmatrix}\right\vert\right\vert\right)
\nonumber\\
\geq&\frac{1}{2}\min\left(1,\frac{1}{\mathrm{tr}(M_{n_k}(z))} \right).\label{GordonInequality}
\end{align}
Assume that $\mathrm{tr}(M_{n_k}(z))$ is bounded. Thus the existence of a sequence $\{n_k\}$ so that
$\alpha(j)=\alpha(j+n_k), 1\leq j\leq n_k$ implies that the sequence $u$ doesn't decay on the right for any initial values $u_0,v_0$ (recall that $u_0, v_0$ cannot both be zero). Note that it is not possible for $v$ to be large at $+\infty$ and for $u$ to decay at $+\infty$, due to the fact that $\mathcal M u= zv$, and that all the $2\times 2$ blocks (the $\Theta_{2k-1}$) that make up $\mathcal M$ have determinant $-1$ and uniformly bounded trace. This all implies that $z$ is not an eigenvalue for the operator $\mathcal E$ and so we proved Proposition \ref{2block}. 

\end{proof}
\begin{proof}[Proof of Proposition \ref{3block}]
We start the same way as the previous proof and arrive again at (\ref{GordonInequality}). This time we assume instead that $\vert\mathrm{tr}(M_{n}(z))\vert>1$ for $n=m_k$, where the $\{m_k\}$ are a subsequence of the $\{n_k\}$. If this subsequence does not exist $\vert\mathrm{tr}(M_{n}(z))\vert$ is bounded and we can simply apply Proposition \ref{2block}. Now apply the matrix equation to 
\[\begin{pmatrix}
u_{-m_k}\\
v_{-m_k}
\end{pmatrix},
\]
and we obtain
\begin{equation*}
\begin{pmatrix}u_{m_k}\\
v_{m_k}
\end{pmatrix}-
\mathrm{tr}(M_{m_k}(z))
\begin{pmatrix}u_{0}\\
v_{0}
\end{pmatrix}+\begin{pmatrix}
u_{-m_k}\\
v_{-m_k}
\end{pmatrix}=0,
\end{equation*}
and given the assumption  $\vert u_0\vert^2+\vert v_0\vert ^2=1$, we must have 

\begin{align*}&\max\left(\left\vert\left\vert\begin{pmatrix}u_{m_k}\\
v_{m_k}
\end{pmatrix}\right\vert\right\vert,\left\vert\left\vert\begin{pmatrix}u_{-m_k}\\
v_{-m_k}
\end{pmatrix}\right\vert\right\vert\right)
\geq\frac{1}{2}.
\end{align*}
Again, it is not possible for $v$ to be large at $+\infty$ or $-\infty$ and for $u$ to decay, since $\mathcal Mu=zv$. Thus the solution $u$ cannot decay at both $\pm\infty$, and we conclude that $z$ is not an eigenvalue for the operator $\mathcal E$. This proves Proposition \ref{3block}.
\end{proof}
\end{section}
\begin{section}{Floquet Theory}
As a first step, we have to say a few words about a Floquet theory for the Gesztesy-Zinchenko transfer matrices. Note that a Floquet theory for the CMV operator is already well understood (see Chapter 11.2 of \cite{Simon}), but using the standard Szeg\H o transfer matrices rather than the Gesztesy-Zinchenko transfer matrices. 

Fix $z\in \partial\D$,  an even period $q\geq 4$, and a \em skew angle \em $\phi\in\partial\D$. We define the \em discriminant \em $\Delta(z_0)=\mathrm{tr} M_q(z_0)$. Our goal is to relate this function with the spectrum of the coresponding periodic extended CMV operator. To be precise, we want to show that for $q$-periodic extended CMV operators the spectrum is exactly the set $z$ where $\vert \Delta(z)\vert\leq 2$. The argument is more or less identical to the one given for Szeg\H o transfer matrices in Chapter 11.2 of \cite{Simon}, but for the readers' convenience we outline a few steps where the transition from Szeg\H o transfer matrices to Gesztesy-Zinchenko transfer matrices isn't as obvious.

 Consider the Floquet matrix given by $\mathcal E_q(\phi)=\mathcal L_q\mathcal M_q(\phi)$, where 

\begin{equation}\label{calM}
\mathcal M_q(\phi)=\begin{pmatrix}
-\alpha(q-1)&&&&\rho(q-1)\phi^{-1}\\
&\Theta_1&&&\\
&&\ddots &&\\
&&& \Theta_{q-3}&\\
\rho(q-1)\phi&&&&\overline{\alpha(q-1)}
\end{pmatrix},
\end{equation}
and 
\begin{equation}\label{CalL}
\mathcal L_q=\begin{pmatrix}
\Theta_0&&&&\\
&\ddots&&&\\
&&\ddots&&\\
&&&\ddots&\\
&&&&\Theta_{q-2}\\
\end{pmatrix}.
\end{equation}
Note that $\mathcal E_q(\phi)$ is a matrix with four nonzero elements on each row, obtained by taking $\mathcal E$, cutting out the $[0,q-1]\times [0,q-1]$ block, and modifying the first two and last two rows in the following way: The top two rows in $\mathcal E$ have one element each cut off in passing to $\mathcal E_q(\phi)$. We shift that element right $q$ places and multiply by $\phi^{-1}$. In the bottom row, we instead shift left by $q$ places and multiply by $\phi$. See Figure 11.3 of \cite{Simon}.

It is then straightforward to verify the following proposition.
\begin{prop}
$z_0\in\partial\D$ is an eigenvalue of $\mathcal E_q(\phi)$ if and only if $\phi$ is an eigenvalue of $M_q(z_0)$.
\end{prop}
\begin{proof}
We simply note that Lemma 2.2 of \cite{GZ} implies that if the solution $u$ to the eigenvalue equation $\mathcal Eu=z_0u$ is a Floquet solution with skew angle $\phi$, then  the corresponding solution $v$ of $\mathcal E^T v=z_0v$ satisfies $\mathcal M u=z_0 v$. It is easy to see that this is equivalent to $\mathcal M_q(\phi) \hat u=z_0\hat v$ for  $\hat u\in\mathbb C^q$ whose terms are $u(1), \ldots u(q)$) and $\hat v\in\mathbb C^q$ whose terms are $v(1), \ldots v(q)$. But then clearly
 \[z_0\phi\hat v=\mathcal M_q(\phi)\phi\hat u,\] and this implies that $\phi v(1)=v(q+1)$.
 Since we already know that $\phi u(1)= u(q+1)$, and since our matrices $M_q(z_0)$ act on vectors of the form \[\begin{pmatrix}
u_n\\v_n
\end{pmatrix},\]
we can conclude that $\phi$ will be an eigenvalue of $M_q(z_0)$ if and only if there exist sequences $u(n),v(n)$ that are each of the form $\pi(n)\phi^n$ with $\pi(n)$ periodic, such that $\mathcal Eu=z_0u$ (and hence $\mathcal E_q\hat u=z_0\hat u$ ) and such that $v$ is associated to $u$ in the manner given by Lemma 2.2 of \cite{GZ}.
\end{proof}
Note that $M_q(z_0)$ is unimodular. Therefore, if $\phi$ is an eigenvalue then $\phi^{-1}$ must be also. We thus see that $z_0$ is an eigenvalue of $\mathcal E_q(\phi)$ if and only if
\begin{equation}\label{FloquetTrace}\Delta(z_0)=\mathrm{Tr}(M_q(z_0))=\phi+\phi^{-1}= 2\cos(k),
\end{equation}
where $\phi=e^{ik}$. We may at this point mimic the standard Floquet Theory of CMV operators to assert that the spectrum of a $q$-periodic CMV operator is precisely the subset of $z\in\partial \D$ when $\vert \Delta(z)\vert\leq 2$.

\end{section}
\begin{section}{Uniform results for the period doubling subshift}
Let us recall our discussion on the period doubling subshift in the introduction. We have an alphabet $\mathcal A=\{a,b\}$. We consider the substitution sequence $S(a)=ab$, $S(b)=aa$, and when we iterate this substitution we obtain a one-sided invariant sequence $u=abaaabab\ldots$. Choose a nonconstant function $f:\mathcal A\to \D$ and then consider the associated subshift $(\Omega,\mathcal T,f)$. We can define for $\omega\in \Omega$ an extended CMV operator $\mathcal E_\omega$ whose corresponding Verblunsky coefficients are  given by $\alpha(n)= f(\omega_n)$.

By construction, we have the recurrence

\begin{align}
\nonumber S^n(a)=&S^{n-1}(a)S^{n-1}(b),\\
S^n(b)=&S^{n-1}(a)S^{n-1}(a).\label{p2.recurrence}
\end{align}

Given an $n$, we define $T_k^a(z)$ for $1\leq k\leq 2^n$ as $T_k(z)$ taking $\alpha(k)=f(S^n(a)_k)$ (To clarify, $S^n(a)_k$ is the $k$th letter of the word $S^n(a)$). Similarly, we define $T_k^b(z)$ for $1\leq k\leq 2^n$ as $T_k(z)$ taking $\alpha(k)=f(S^n(b)_k)$. Following from (\ref{p2.recurrence}) and the discussion on the period doubling sequence earlier in this section, let us define also 
\begin{align}\nonumber
M_{(a),n}(z)=& \prod_{k=2^n}^{1} T_k^a(z),\\
M_{(b),n}(z)=&\prod_{k=2^n}^{1} T_k^b(z).\label{M_n^a}
\end{align}
We will often suppress the dependence on $z$ for notational convenience. We verify using (\ref{p2.recurrence}) that for $n\geq 1$ we have
\begin{align}
M_{(a),n+1}=\prod_{k=2^{n+1}}^{1} T_k^a(z)=\prod_{k=2^n}^{1} T_k^b(z)\prod_{k=2^n}^{1} T_k^a(z) =M_{(b),n}M_{(a),n},\nonumber \\
M_{(b),n+1}=\prod_{k=2^{n+1}}^{1} T_k^b(z)=\prod_{k=2^n}^{1} T_k^a(z)\prod_{k=2^n}^{1} T_k^a(z)=M_{(a),n}M_{(a),n}.
\end{align}

\begin{lemma}\label{Blemma}
 $\mathrm{Tr}(M_{(b),n}M_{(a),n}^{-1})=\frac{2(\mathrm{Re}(-f(a)\overline{f(b)})+1)}{\sqrt{1-\vert f(a)\vert^2}\sqrt{1-\vert f(b)\vert^2}}\geq 2$ for $n\geq 1$.
 \end{lemma}
\begin{proof}
First, using the fact that $\vert f(a)\vert, \vert f(b)\vert\leq 1$, we observe
\[\frac{2(\mathrm{Re}(-f(a)\overline{f(b)})+1)}{\sqrt{1-\vert f(a)\vert^2}\sqrt{1-\vert f(b)\vert^2}}\geq  \frac{2(1-\vert f(a)\vert \cdot\vert f(b)\vert)}{\sqrt{1-\vert f(a)\vert^2}\sqrt{1-\vert f(b)\vert^2}}\geq 2.\]

We can compute that the lemma holds for $n=1$:
\begin{align*}
M_{(b),1}M_{(a),1}^{-1}=&T_2^b(z)T_1^b(z)T_1^a(z)^{-1}T_2^a(z)^{-1}\\
=&T_2^b(z)T_2^a(z)^{-1},\text{  since $T_1^a(z)=T_1^b(z)$}\\
=&\frac{1}{\sqrt{1-\vert f(a)\vert^2}}\begin{pmatrix}
-f(a)&1\\
1&-\overline{f(a)}
\end{pmatrix}
\frac{-1}{\sqrt{1-\vert f(b)\vert^2}}\begin{pmatrix}
-\overline{f(b)}&-1\\
-1&-f(b)
\end{pmatrix}\\
=&\frac{1}{\sqrt{1-\vert f(a)\vert^2}}\frac{1}{\sqrt{1-\vert f(b)\vert^2}}\begin{pmatrix}
-f(a)\overline{f(b)}+1 & f(b)-f(a)\\
\overline{f(b)}-\overline{f(a)}&-f(b)\overline{f(a)}+1
\end{pmatrix}\\
\end{align*}
Thus we have
\begin{align*}
\mathrm{Tr}(M_{(b),1}M_{(a),1}^{-1})=&\frac{2(\mathrm{Re}(-f(a)\overline{f(b)})+1)}{\sqrt{1-\vert f(a)\vert^2}\sqrt{1-\vert f(b)\vert^2}}.
\end{align*}
 Note that for $n\geq 1$

\begin{align*} 
\mathrm{Tr}(M_{(b),n+1}M_{(a),n+1}^{-1})=& \mathrm{Tr}(M_{(a),n}M_{(a),n}M_{(a),n}^{-1}M_{(b),n}^{-1}) \\
=& \mathrm{Tr}(M_{(a),n}M_{(b),n}^{-1})\\
=&\mathrm{Tr}(M_{(b),n}M_{(a),n}^{-1}).
\end{align*}
The last equality following from the fact that $M_{(a),n}M_{(b),n}^{-1}$ has determinant $1$. The lemma then follows from a simple induction argument. 
\end{proof}
Let us define $x_n=\mathrm{Tr}(M_{(a),n})$, $y_n=\mathrm{Tr}(M_{(b),n})$, and $B=\mathrm{Tr}(M_{(b),1}M_{(a),1}^{-1})$. We instead write $x_n(z), y_n(z)$ when we wish to emphasize the dependence on $z$. We will perform an analysis on the trace map dynamics of $M$, a central tool in many previous treatments of operators generated by subshifts. See \cite{Peyriere} for a survey of the method.

\begin{prop}
For $n\geq 1$ the sequences $x_n, y_n$ are real and obey the recursion
\begin{align}
x_{n+1}=x_ny_n-B, \label{PDRecursion_x}\\
y_{n+1}=x_n^2-2.\label{PDRecursion_y}
\end{align}

\end{prop}
\begin{proof}
Note the following standard identity for unimodular $2\times 2$ matrices $M,N$ (for instance, refer to (17) of \cite{BIST})
\begin{equation}\label{traceidentity}
\mathrm{tr}(MN)=\mathrm{tr}(M)\mathrm{tr}(N)-\mathrm{tr}(MN^{-1}).
\end{equation}
Using Lemma \ref{Blemma}, we know that if we set $M=M_{(b),n}, N=M_{(a),n}$ we get the first equation, and if we set $M=N=M_{(a),n}$ we get the second.

We calculate that $x_1, y_1$ are both real, and as a consequence $x_n,y_n$ is real for all $n\geq 1$.
\end{proof}

\begin{defn}
$\mathcal U$ is the set of all points $(x,y)\in \R^2$ such that there is $n_0$ such that if $x_0=x, y_0=y$, for all $n\geq n_0$, $\vert x_n\vert >2 $. We call the interior of $\mathcal U$ the set of \em unstable points\em.
\end{defn}

We note here that the interior of $\mathcal U$ is certainly nonempty, since by (\ref{PDRecursion_x}), (\ref{PDRecursion_y}) all points with $x$ and $y$ sufficiently large (say, both larger than $B$) are in $\mathcal U$. We will also remark that in the Schr\"odinger setting, it is known that $\mathcal U$ is an open set (Lemma 2 of \cite{BBG}). That lemma doesn't directly apply in our setting unless $B= 2$. Nevertheless, in our paper the question of whether $\mathcal U$ is open or not only affects the proof of Proposition \ref{tracebounds} below, and as the reader will see it suffices to consider the interior of $\mathcal U$.

\begin{prop} \label{mathcal U}
If $(x_0(z), y_0(z))$ is an unstable point, then $z$ is in the resolvent set of the operator $\mathcal E$.
\end{prop}
\begin{proof}
In the Schr\"odinger setting, the proof was given in \cite{Bellisard}. Let $z\in\partial \D$, and let us label $\hat{\mathcal E}(n)$ as the periodic approximation of $\mathcal E$, generated by Verblunsky coefficients $\hat \alpha(j)= \alpha(j')$, where $j'\equiv j\mod 2^n$ and $j'\in [0,2^n)$. In this case, $\vert x_n(z)\vert>2$ means that $z$ is in a spectral gap of $\hat{\mathcal E}(n)$. Thus if $(x_N(z), y_N(z))$ is unstable for some $N$, there is some neighborhood $U$ of $z$ that is in a spectral gap of a sequence $\hat{\mathcal E}(t)$ for $t\geq N$ of periodic approximations of $\mathcal E$. We conclude that $z$ is also in a spectral gap of $\mathcal E$.
\end{proof}
We now establish trace bounds for points in the spectrum.
\begin{prop}\label{tracebounds}
If $z\in\partial \D$ is in the spectrum of $\mathcal E$, then for any integer $n\geq 1$ either $\vert x_n(z)\vert \leq B$ or $\vert x_{n+1}(z)\vert \leq B$.
\end{prop}
\begin{proof}
Recall that from Lemma \ref{Blemma}, $B\geq 2$. Assume instead that for some $n$ we have both $\vert x_{n}(z)\vert >B$ and $\vert x_{n+1}(z)\vert >B$. By (\ref{PDRecursion_y}) we must also have $y_{n+1}(z)>B$. But this implies, by (\ref{PDRecursion_x}) that $\vert x_{n+2}(z)\vert >B$. Thus it is an easy induction argument to show that $\vert x_{m}(z)\vert >B\geq 2$ for all $m\geq n$. This implies that $(x_0(z),y_0(z))\in\mathcal U$.

Furthermore, let us consider the open subsets $D_+, D_-$ of $\mathbb R^2$, defined as $D_\pm=\{\pm x>B, y> 2\}$. It is clear that if for some $n$ $(x_n,y_n) \in D_+\cup D_-$ then $(x_0,y_0)\in\mathcal U$. Since $x_n, y_n$ are analytic functions of $x_0, y_0$ for fixed $n$, if we choose $(x_0,y_0)$ such that $(x_n,y_n)\in D_+\cup D_-$, then every point in some open neighborhood $S$ of $(x_0,y_0)$ must have their $n$th images in $D_+\cup D_-$, and thus every point in $S$ must be unstable. Thus we can show that $(x_0(z),y_0(z))$ is in the interior of $\mathcal U$, and so by Proposition \ref{mathcal U}, we conclude $z$ must be in the resolvent set of $\mathcal E$.
\end{proof}
\begin{proof}[Proof of Theorem \ref{PDtheorem}]
Absolutely continuous spectrum is ruled out in \cite{DL_Szego}. In this section we have already established all the tools necessary in the CMV setting, and so we can now easily replicate the combinatorial analysis in \cite{Damanik.doubling} to rule out point spectrum. Specifically,
our Propositions \ref{tracebounds}, \ref{3block} and \ref{2block} replace Propositions 2.2, 2.4a, and 2.4b respectively in that paper. Since the subshift we consider is identical, what follows is simply an extended case analysis on the given element $\omega\in\Omega$ which works exactly the same way as in \cite{Damanik.doubling}.
\end{proof}
As a direction for future research, it will be interesting to replicate the analysis in Section 3 of \cite{BBG} where they proved several results about the spectral gaps for the Schr\"odinger operator with period doubling potential, in addition to proving that the spectrum is uniformly singular continuous. In particular, they were able to determine which gaps were open, and to calculate the rate of growth of the gaps in terms of the absolute value of the potential.

 Their results rely on exploiting trace map dynamics for Schr\"odinger transfer matrices. The analysis in our CMV context will be more complicated, since our trace map equations (\ref{PDRecursion_x}), (\ref{PDRecursion_y}) have an additional parameter $B$, and as a result the nature of the fixed points, et cetera will not be the same. For example, in general it is not clear in our context that $\mathcal U$ is an open set, a fact that is true in the Schr\"odinger setting and an essential first step for many of the theorems in \cite{BBG}. 
\end{section}

\begin{section}{Almost sure results for subshifts generated by Sturmian sequences and codings of rotations}

In this section, we will consider Verblunsky coefficients generated by subshifts of \em Sturmian sequences \em and \em codings of rotations\em. To define a Sturmian sequence, recall that we choose a rotation angle $\theta\in [0,1)$, and let $a_n$ denote the coefficients in its continued fraction expansion. 

Let $f$ be a nonconstant function from $\{0,1\}$ to $\D$. Recall that we define our Verblunsky coefficients of our extended CMV matrix $\mathcal E_{\theta, \beta}$ by 
\[\alpha_{\theta, \beta}(n)=f(\chi_{[1-\theta,1)}(n\theta+\beta\mod 1)).\]
 
 As for codings of rotations, we instead choose an interval $I\subseteq \R/\Z$, and define our Verblunsky coefficients by
 
 \[\alpha_{\theta, \beta}(n)=f(\chi_{I}(n\theta+\beta\mod 1)).\] 

 In both cases, by standard results in ergodic theory (see for instance Theorem 10.5.21 of \cite{Simon}), the spectrum $\Sigma_\theta$ of $\mathcal E_{\theta, \beta}$ is independent of $\beta$.

 The proofs of Theorems \ref{Kaminaga0}, \ref{Kaminaga1} and \ref{Kaminaga2} are a close analogue of the proofs in \cite{Kaminaga}, which we sketch here for the reader's convenience. We first address the preliminaries, Lemmas 1,2,3 and 4 in that paper. Their Lemma 2 is proven in the CMV context as Theorem 10.16.1 in \cite{Simon} (that is, the fact that the spectrum is a.e independent of  $\omega\in\Omega$, and that we have a.e absence of point spectrum). This in turn implies their Lemma 1, which asserts that absence of point spectrum for a positive measure set of $\beta$ implies absence of point spectrum for a full measure set of $\beta$. Their Lemma $3$ is simply the Gordon lemma, so we can equivalently use our Proposition \ref{3block}. Finally, their Lemma 4 is a simple calculation. We have to make a small change to that Lemma 4, however. Rather than taking $\limsup$s over $n\in \mathbb N$, we only take $\limsup$s over $n$ in $\tilde{\mathbb N}$, (recall that $\tilde{\mathbb N}$ is the subset of $\mathbb N$ so that $q_n$ is even), and we thus obtain instead
 
\begin{equation}\label{lemma4} \limsup_{n\in \tilde{\mathbb N}} q_{n+1}/q_n\geq\limsup_{n\in \tilde{\mathbb N}} a_n.
\end{equation}
Again, absence of absolutely continuous spectrum was demonstrated in \cite{DL_Szego}.

 \begin{proof}[Proof of Theorem \ref{Kaminaga2}]
 
Where $\alpha_\beta$ refers to the Verblunsky coefficients generated using the phase $\beta$, let us define 
\[ E(n):=\{\beta\vert \alpha_\beta(j+q_n)=\alpha_\beta(j-q_n)=\alpha_\beta(j), 1\leq j\leq q_n\},\]
so $E(n)$ is the set of phases for which Proposition \ref{3block} is satisfied for $q_n$. It suffices to show \[\mu\left(\limsup_{n\in \tilde{\mathbb N}} E(n)\right)>0,\]
where $\mu$ is Lebesgue measure on $[0,1)$. Recall that $I$ is the interval for our rotation coding sequence, and let us denote $\beta_1,\beta_2$ as the two endpoints of $I$. Then where $\Phi$ is the canonical projection from $\mathbb R$ to $\R/\Z$, and $\vert \cdot \vert_1$ denotes distance from $0$ in $\R/\Z$, we define sets

\begin{equation}\label{Kaminaga.8}
E_i(n)=\{\beta\vert \min_{1\leq j\leq q_n} \vert \Phi(j\theta)+\beta-\beta_i\vert_1>\vert q_n\theta-p_n\vert\}, i=1,2.
\end{equation}

We can calculate that

\begin{equation}\label{Kaminaga.9}
\vert (\Phi((j+q_n)\theta)+\beta)-(\Phi(j\theta)+\beta)\vert_1=\vert q_n \theta-p_n\vert,
\end{equation}
 By (\ref{Kaminaga.8}) and (\ref{Kaminaga.9}) we then have $E_1(n)\cap E_2(n)\subset E(n)$. After some calculations, this gets us 
 \[ \mu(E(n))\geq 1- 4\frac{q_n}{q_{n+1}},\]
 and therefore
 
\[ \limsup_{n\in \tilde{\mathbb N}}\mu(E(n))\geq 1-\frac{4}{\limsup_{n\in\tilde{\mathbb N}}\frac{q_{n+1}}{q_n}}.\]
And so with (\ref{lemma4}) and the hypothesis that $\limsup_{n\in \tilde{\mathbb N}} a_n\geq 5$, we are done.
 \end{proof}
 \begin{proof}[Proof of Theorem \ref{Kaminaga1}]
 We repeat the steps of the previous proof, but this time we set $\beta_1=1-\theta$ and $\beta_2=0$. We then perform similar computations to obtain
 \begin{equation}\label{mu(E(n))}
 \mu(E(n))\geq 1-2(q_n+1)\vert q_n\theta-p_n\vert,
 \end{equation}
 and then, using (5) from \cite{Kaminaga}  
 \[\mu(E(n))\geq 1-\frac{2q_n+2}{q_{n+1}}.\]
 
We can then conclude that $\limsup_{n\in\tilde{\mathbb N}}\mu(E(n))$ is positive if we apply (\ref{lemma4}) and the hypothesis $\limsup_{n\in \tilde{\mathbb N}} a_n\geq 3$.
 \end{proof}

\begin{proof}[Proof of Theorem \ref{Kaminaga0}]
By (7) of \cite{Kaminaga}, 
\[\lim_{n\to\infty} \frac{q_{n+1}}{q_n}=\frac{2}{\sqrt{5}-1}.\]
But this certainly implies
\[\limsup_{n\in\tilde{\mathbb N}} \frac{q_{n+1}}{q_n}=\frac{2}{\sqrt{5}-1}.\]

Using this equation to replace (7) (that is, (7) in their paper, not ours) in the proof of Lemma 5 in \cite{Kaminaga}, we can assert that 
\[\limsup_{n\in\tilde{\mathbb N}} q_n\vert q_n\theta-p_n\vert=\frac{1}{\sqrt{5}}.\]
We plug this in to (\ref{mu(E(n))}), and then we are done.
\end{proof}
As a final remark, we will note that \cite{DKL} proved singular continuous spectrum for Schr\"odinger operators for all $\beta$ and without any restrictions on the $a_n$'s. Unfortunately, there are at least two obstacles to applying that argument in our setting directly. First, our Gordon lemma for the CMV operator only works for even jumps, and so to apply the argument in \cite{DKL} directly we would require consecutive $q_n$ to be even, which is impossible. Furthermore, we lack a Floquet theory for solutions of odd-periodic Gesztesy-Zinchenko transfer matrices, which we require to establish trace bounds. 
\end{section}

\bibliographystyle{alpha}   % this means that the order of references
			    % is determined by the order in which the
			    % \cite and \nocite commands appear
\bibliography{mybib}

\begin{thebibliography}{CGMV10}

\bibitem[ADZ13]{Avila-Damanik-Zhang}
Artur Avila, David Damanik, and Zhenghe Zhang.
\newblock Singular density of states measure for subshift and quasi-periodic
  {S}chr\"odinger operators.
\newblock arXiv:1304.0519, 2013.

\bibitem[BBG91]{BBG}
J.~Bellisard, A.~Bovier, and J.-M. Ghez.
\newblock Spectral properties of a tight-binding {H}amiltonian with period
  doubling potential.
\newblock {\em Communications in Mathematical Physics}, 135:379--399, 1991.

\bibitem[Bel90]{Bellisard}
J.~Bellisard.
\newblock Spectral properties of {S}chr\"odinger's operator with a
  {T}hue-{M}orse potential.
\newblock In P.~Waldshmidt~M. Luck, J.M.-Moussa, editor, {\em From number
  theory to physics}, volume~47 of {\em Springer Proceedings in Physics}.
  Springer, Berlin, Heidelberg, New York, 1990.

\bibitem[BIST89]{BIST}
J.~Bellissard, B.~Iochum, E.~Scoppola, and D.~Testard.
\newblock Spectral properties of one-dimensional quasicrystals.
\newblock {\em Communications in Mathematical Physics}, 125:527--543, 1989.

\bibitem[CFKS08]{Cycon}
Hans Cycon, Richard Froese, Werner Kirsch, and Barry Simon.
\newblock {\em Schr\"odinger {O}perators with Applications to quantum mechanics
  and global geometry}.
\newblock Springer, 2008.

\bibitem[CGMV10]{CMVG}
Mar\'ia-Jos\'e Cantero, F.~Alberto Gr\"unbaum, Leandro Moral, and Luis
  Vel\'azquez.
\newblock Matrix-valued {S}zeg\"o polynomials and quantum random walks.
\newblock {\em Communications in {P}ure and {A}pplied {M}athematics},
  LXIII:0464--0507, 2010.

\bibitem[CMV03]{CMV}
Mar\'ia-Jos\'e Cantero, Leandro Moral, and Luis Vel\'azquez.
\newblock Five-diagonal matrices and zeros of orthogonal polynomials on the
  unit circle.
\newblock {\em Linear Algebra and its Applications}, 362:29--56, 2003.

\bibitem[Dah11]{Dahl}
Janine Dahl.
\newblock {\em The spectrum of the off-diagonal {F}ibonacci operator}.
\newblock PhD thesis, Rice University, 2011.

\bibitem[Dam00]{Damanik.Gordon}
David Damanik.
\newblock Gordon-type arguments in the spectral theory of one-dimensional
  quasicrystals.
\newblock In M.~Baake and R.~V. Moody, editors, {\em Directions in Mathematical
  Quasicrystals}, volume~13 of {\em CRM Monograph Series}. American
  Mathematical Society, Providence, 2000.

\bibitem[Dam01]{Damanik.doubling}
{D}avid Damanik.
\newblock Uniform singular continuous spectrum for the period doubling
  hamiltonian.
\newblock {\em Annales Henri Poincar\'e}, 2:101--108, 2001.

\bibitem[Dam07]{Damanik.subshiftexpo}
{D}avid Damanik.
\newblock Strictly ergodic subshifts and associated operators.
\newblock In Fritz Gesztesy, Percy Deift, Cherie Galvez, Peter Perry, and
  Wilhelm Schlag, editors, {\em Spectral theory and mathematical physics: a
  Festschrift in honor of Barry Simon's 60th birthday}, volume 76.2 of {\em
  Proceedings of Symposia in Pure Mathematics}. American Mathematical Society,
  2007.

\bibitem[DEG12]{Damanik-Embree-Gorodetski}
David Damanik, Mark Embree, and Anton Gorodetski.
\newblock Spectral properties of {S}chr\"odinger operators arising in the study
  of quasicrystals.
\newblock arXiv:1210.5753, 2012.

\bibitem[DG11]{Damanik-Gorodetski}
{D}avid Damanik and Anton Gorodetski.
\newblock Spectral and quantum dynamical properties of the weakly coupled
  {F}ibonacci {H}amiltonian.
\newblock {\em Communications in {M}athematical {P}hysics}, 305:221--277, 2011.

\bibitem[DKL00]{DKL}
{D}avid Damanik, {R}owan {K}illip, and Daniel {L}enz.
\newblock Uniform spectral properties of one-dimensional quasicrystals,
  {I}{I}{I}. $\alpha$-continuity.
\newblock {\em Communications in {M}athematical {P}hysics}, 212:191--204, 2000.

\bibitem[DL07]{DL_Szego}
{D}avid Damanik and Daniel {L}enz.
\newblock Uniform {S}zeg{\H o} cocycles over strictly ergodic subshifts.
\newblock {\em Journal of Approximation Theory}, 144:133--138, 2007.

\bibitem[DMY13a]{DMY1}
David Damanik, Paul Munger, and William Yessen.
\newblock Orthogonal polynomials on the unit circle with {F}ibonacci
  {V}erblunsky coefficients, {I}. {T}he essential support of the measure.
\newblock {\em Journal of Approximation Theory}, 173:56--88, 2013.

\bibitem[DMY13b]{DMY2}
David Damanik, Paul Munger, and William Yessen.
\newblock Orthogonal polynomials on the unit circle with {F}ibonacci
  {V}erblunsky coefficients, {I}{I}. {A}pplications.
\newblock {\em Journal of Statistical Physics}, 153(2):339--362, 2013.

\bibitem[Gor76]{Gordon}
Alexander Gordon.
\newblock The point spectrum of the one-dimensional {S}chr\"odinger operator.
\newblock {\em Uspekhi {M}at. {N}auk}, 31:4(190):257--258, 1976.

\bibitem[GZ06]{GZ}
Fritz Gesztesy and Maxim Zinchenko.
\newblock Weyl{-}{T}itchmarsh theory for {C}{M}{V} operators associated with
  orthogonal polynomials on the unit circle.
\newblock {\em Journal of Approximation Theory}, 139:172--213, 2006.

\bibitem[Kam96]{Kaminaga}
Masahiro Kaminaga.
\newblock Absence of point spectrum for a class of discrete {S}chr\"odinger
  operators with quasiperiodic potential.
\newblock {\em Forum Mathematicum}, 8:63--69, 1996.

\bibitem[Ong12]{Ong-LP}
Darren~C. Ong.
\newblock Limit-periodic {V}erblunsky coefficients for orthogonal polynomials
  on the unit circle.
\newblock {\em Journal of Mathematical Analysis and Applications}, 394(2):633
  -- 644, 2012.

\bibitem[Ong13]{Ong-QP}
Darren~C. Ong.
\newblock Orthogonal polynomials on the unit circle with quasiperiodic
  {V}erblunsky coefficients have generic purely singular continuous spectrum.
\newblock {\em Discrete and Continuous Dynamical Systems}, pages 605--609,
  2013.
\newblock {P}roceedings of the 9th {A}{I}{M}{S} {C}onference Special Issue.

\bibitem[Pey95]{Peyriere}
J.~Peyri\'ere.
\newblock Trace maps.
\newblock In Fran\c{c}oise Axel and Denis Gratias, editors, {\em Beyond
  Quasicrystals}. Springer, 1995.

\bibitem[Sim04]{Simon}
Barry Simon.
\newblock {\em Orthogonal {P}olynomaials on the {U}nit {C}ircle}.
\newblock American {M}athematical {S}ociety, 2004.

\bibitem[S{\H u}t87]{Suto87}
Andr\'as S{\H u}to.
\newblock The spectrum of a quasiperiodic {S}chr\"odinger operator.
\newblock {\em Communications in Mathematical Physics}, 111:409--415, 1987.

\bibitem[S{\H u}t89]{Suto89}
Andr\'as S{\H u}to.
\newblock Singular continuous spectrum on a set of zero {L}ebesgue measure for
  the {F}ibonacci {H}amiltonian.
\newblock {\em Journal of Statistical Physics}, 56:525--521, 1989.

\bibitem[S{\H u}t95]{Suto95}
Andr\'as S{\H u}to.
\newblock Schr\"odinger difference equation with deterministic ergodic
  potentials.
\newblock In Fran\c{c}oise Axel and Denis Gratias, editors, {\em Beyond
  Quasicrystals}. Springer, 1995.

\bibitem[Yes13]{Yessen.tridiagonal}
William Yessen.
\newblock Spectral analysis of tridiagonal fibonacci hamiltonians.
\newblock {\em Journal of Spectral Theory}, 3(1):101--128, 2013.

\end{thebibliography}
\end{document}